\documentclass[copyright,creativecommons]{eptcs}
\providecommand{\event}{QPL 2013}
\usepackage{breakurl}

\usepackage{amsmath,amsthm,amssymb}
\usepackage{xspace,xypic,enumerate,color}
\usepackage{tikzfig}
\xyoption{curve}

\tikzstyle{every picture}=[baseline=-0.25em,shorten <=-0.1pt]
\tikzstyle{dotpic}=[scale=0.5]

\tikzstyle{braceedge}=[decorate,decoration={brace,amplitude=1mm,raise=-1mm}]
\tikzstyle{left hook arrow}=[left hook-latex]
\tikzstyle{right hook arrow}=[right hook-latex]

\tikzstyle{dot}=[inner sep=0.7mm,minimum width=0pt,minimum height=0pt,fill=black,draw=black,shape=circle]

\tikzstyle{black dot}=[dot]
\tikzstyle{white dot}=[dot,fill=white]
\tikzstyle{gray dot}=[dot,fill=gray!40!white]

\tikzstyle{alt white dot}=[white dot,label={[xshift=3mm,yshift=-0.05mm,font=\tiny]left:$*$}]
\tikzstyle{alt gray dot}=[gray dot,label={[xshift=3mm,yshift=-0.05mm,font=\tiny]left:$*$}]

\tikzstyle{white norm}=[rectangle,fill=white,draw=black,minimum height=2mm,minimum width=2mm,inner sep=0pt,font=\small]
\tikzstyle{gray norm}=[white norm,fill=gray!40!white]
\tikzstyle{black norm}=[white norm,fill=black]

\tikzstyle{arrs}=[-latex,font=\small,auto]
\tikzstyle{arrow plain}=[arrs]
\tikzstyle{arrow dashed}=[dashed,arrs]
\tikzstyle{arrow bold}=[very thick,arrs]
\tikzstyle{arrow hide}=[draw=white!0,-]
\tikzstyle{arrow reverse}=[latex-]
\tikzstyle{cdnode}=[]

\tikzstyle{wide point}=[fill=white,draw=black,shape=isosceles triangle,shape border rotate=90,isosceles triangle stretches=true,inner sep=1pt,minimum width=1.5cm,minimum height=5mm]
\tikzstyle{wide copoint}=[fill=white,draw=black,shape=isosceles triangle,shape border rotate=-90,isosceles triangle stretches=true,inner sep=1pt,minimum width=1.5cm,minimum height=4mm]
\tikzstyle{very wide copoint}=[fill=white,draw=black,shape=isosceles triangle,Shape border rotate=-90,isosceles triangle stretches=true,inner sep=1pt,minimum width=2.5cm,minimum height=4mm]
\tikzstyle{very wide empty copoint}=[draw=black,shape=isosceles triangle,shape border rotate=-90,isosceles triangle stretches=true,inner sep=1pt,minimum width=2.5cm,minimum height=4mm]
\tikzstyle{symm}=[ultra thick,shorten <=-1mm,shorten >=-1mm]

\tikzstyle{square box}=[rectangle,fill=white,draw=black,minimum height=3mm,minimum width=3mm,font=\small]
\tikzstyle{square gray box}=[rectangle,fill=gray!30,draw=black,minimum height=6mm,minimum width=6mm]
\tikzstyle{point}=[regular polygon,regular polygon rotate=180,regular polygon sides=3,draw=black,scale=0.75,inner sep=-0.5pt,minimum width=.6cm,fill=white]
\tikzstyle{copoint}=[regular polygon,regular polygon sides=3,draw=black,scale=0.75,inner sep=-0.5pt,minimum width=1cm,fill=white]
\tikzstyle{gray point}=[point,fill=gray!40!white]
\tikzstyle{gray copoint}=[copoint,fill=gray!40!white]

\tikzstyle{diredge}=[->]
\tikzstyle{rdiredge}=[<-]
\tikzstyle{dashed edge}=[dashed]

\tikzstyle{cross}=[preaction={draw=white, -, line width=3pt}]

\tikzstyle{black point}=[regular polygon,regular polygon rotate=180,regular polygon sides=3,draw=black,scale=0.25,inner sep=-0.5pt,minimum width=1cm,fill=black]

\tikzstyle{grey point}=[regular polygon,regular polygon rotate=180,regular polygon sides=3,draw=black,scale=0.25,inner sep=-0.5pt,minimum width=1cm,fill=gray!40!white]

\tikzstyle{white point}=[regular polygon,regular polygon rotate=180,regular polygon sides=3,draw=black,scale=0.25,inner sep=-0.5pt,minimum width=1cm,fill=white]

\tikzstyle{wide white point}=[regular polygon,regular polygon rotate=180,regular polygon sides=3,draw=black,xscale=0.25,yscale=0.15,inner sep=-0.5pt,minimum width=15mm,fill=white]

\tikzstyle{wide white copoint}=[regular polygon,regular polygon rotate=0,regular polygon sides=3,draw=black,xscale=0.25,yscale=0.15,inner sep=-0.5pt,minimum width=15mm,fill=white]

\tikzstyle{black copoint}=[regular polygon,regular polygon rotate=0,regular polygon sides=3,draw=black,scale=0.25,inner sep=-0.5pt,minimum width=1cm,fill=black]

\tikzstyle{grey copoint}=[regular polygon,regular polygon rotate=0,regular polygon sides=3,draw=black,scale=0.25,inner sep=-0.5pt,minimum width=1cm,fill=gray!40!white]

\tikzstyle{white copoint}=[regular polygon,regular polygon rotate=0,regular polygon sides=3,draw=black,scale=0.25,inner sep=-0.5pt,minimum width=1cm,fill=white]

\newcommand{\subsystemcolour}{gray!40!white}
\pgfkeys{/tikz/.cd, subsystemcolour/.store in=\subsystemcolour}
\tikzstyle{subsystem}=[postaction={decorate,decoration={markings,
    mark=at position .4 with {\arrow[rotate=180, fill=\subsystemcolour]{triangle 60}}}}]
\tikzstyle{None}=[circle, fill=white, inner sep=0pt]

\newcommand{\dotunit}[1]{%
\,\begin{tikzpicture}[dotpic,yshift=1.5mm]
\node [#1] (a) at (0,-0.25) {}; 
\draw [diredge] (a)--(0,0.2);
\end{tikzpicture}\,}

\newcommand{\dotmult}[1]{%
\,\begin{tikzpicture}[dotpic]
	\node [#1] (a) {};
	\draw [diredge] (a) -- (90:0.55);
	\draw [diredge] (a) (-45:0.6) -- (a);
	\draw [diredge] (a) (-135:0.6) -- (a);
\end{tikzpicture}\,}

\newcommand{\dotaction}[1]{%
\,\begin{tikzpicture}[dotpic,yshift=0.5mm]
	\node [#1] (a) {};
	\draw [diredge] (-90:0.55)--(a);
	\draw [diredge] (a) -- (45:0.6);
	\draw [rdiredge] (a) -- (135:0.6);
\end{tikzpicture}\,}
\newcommand{\dotcoaction}[1]{%
\,\begin{tikzpicture}[dotpic]
	\node [#1] (a) {};
	\draw [diredge] (a) -- (90:0.55);
	\draw [diredge] (a) (-45:0.6) -- (a);
	\draw [rdiredge] (a) (-135:0.6) -- (a);
\end{tikzpicture}\,}

\newcommand{\dotdualmult}[1]{%
\!\begin{tikzpicture}[dotpic]
		\node [style=white dot] (0) at (0, 0.3) {};
		\node [style=none] (1) at (-0.5, -0.4) {};
		\node [style=none] (2) at (0.5, -0.4) {};
		\node [style=none] (3) at (0, 0.8) {};
		\draw [style=diredge] (3.center) to (0);
		\draw [style=diredge, in=15, out=-30, looseness=1.50] (0) to (1.center);
		\draw [style=diredge, in=165, out=-150, looseness=1.50] (0) to (2.center);
\end{tikzpicture}\!}

\newcommand{\dotcap}[1]{%
\,\begin{tikzpicture}[dotpic,yshift=2.5mm]
	\node [#1] (a) at (0,0) {};
	\draw [bend right,diredge] (a) to (-0.4,-0.6);
	\draw [bend left,diredge] (a) to (0.4,-0.6);
\end{tikzpicture}\,}

\newcommand{\dotnorm}[1]{%
\,\begin{tikzpicture}[dotpic,yshift=0.4mm]
		\node [style=none] (0) at (0, -0.4) {};
		\node [style=#1] (1) at (0, -0) {};
		\node [style=none] (2) at (0, 0.5) {};
		\draw (0.center) to (1);
		\draw [style=diredge] (1) to (2.center);
\end{tikzpicture}\,}
\newcommand{\dotconorm}[1]{%
\,\begin{tikzpicture}[dotpic,yshift=0.4mm]
		\node [style=none] (0) at (0, -0.4) {};
		\node [style=white norm] (1) at (0, 0.1) {};
		\node [style=none] (2) at (0, 0.5) {};
		\draw [style=diredge] (1) to (0.center);
		\draw (2.center) to (1);
\end{tikzpicture}\,}

\newcommand{\splittriangle}{%
\,\begin{tikzpicture}[dotpic]
		\node [style=wide white point] (0) at (0, 0) {};
                \draw [style=diredge] (0.north east) to ([yshift=3mm]0.north east);
                \draw [style=diredge] ([yshift=3mm]0.north west) to (0.north west);
                \draw [style=diredge] ([yshift=-3mm]0.south) to (0.south);
\end{tikzpicture}\,}

\newcommand{\splitcotriangle}{%
\,\begin{tikzpicture}[dotpic]
		\node [style=wide white copoint] (0) at (0, 0) {};
                \draw [style=diredge] (0.south east) to ([yshift=-3mm]0.south east);
                \draw [style=diredge] ([yshift=-3mm]0.south west) to (0.south west);
                \draw [style=diredge] ([yshift=3mm]0.north) to (0.north);
\end{tikzpicture}\,}

\newcommand{\blackunit}{\dotunit{dot}}

\newcommand{\blackmult}{\dotmult{dot}}

\newcommand{\blackaction}{\dotaction{dot}}
\newcommand{\blackcoaction}{\dotcoaction{dot}}
\newcommand{\blacknorm}{\dotnorm{black norm}}

\newcommand{\blackcap}{\dotcap{dot}}

\newcommand{\whiteunit}{\dotunit{white dot}}

\newcommand{\whitemult}{\dotmult{white dot}}

\newcommand{\whitecap}{\dotcap{white dot}}

\newcommand{\whitenorm}{\dotnorm{white norm}}

\newcommand{\whiteaction}{\dotaction{white dot}}
\newcommand{\whitecoaction}{\dotcoaction{white dot}}

\newcommand{\graynorm}{\dotnorm{gray norm}}

\let\olddagger\dagger
\renewcommand{\dagger}{\ensuremath{\olddagger}\xspace}

\newcommand{\cat}[1]{\ensuremath{\mathbf{#1}}\xspace}
\newcommand{\FHilb}{\cat{FHilb}}
\newcommand{\Hilb}{\cat{Hilb}}
\newcommand{\Rel}{\cat{Rel}}

\newcommand{\CP}{\ensuremath{\mathrm{CP}\xspace}}
\newcommand{\CPs}{\ensuremath{\CP^*}\xspace}
\newcommand{\CPsn}{\ensuremath{\CP^*_{\text{n}}}\xspace}
\newcommand{\CPM}{\ensuremath{\mathrm{CPM}}\xspace}
\newcommand{\Split}[1]{\ensuremath{\mathrm{Split}_{#1}}\xspace}
\newcommand{\Splitd}{\ensuremath{\mathrm{Split}^\dag}\xspace}
\newcommand{\V}{\cat{V}}
\newcommand{\C}{\ensuremath{\mathbb{C}}}
\newcommand{\inprod}[2]{\ensuremath{\langle #1\,|\,#2 \rangle}}
\DeclareMathOperator{\Tr}{Tr}
\newcommand{\id}[1][]{\ensuremath{1_{#1}}}
\DeclareMathOperator{\Ob}{Ob}
\DeclareMathOperator{\Mor}{Mor}
\DeclareMathOperator{\dom}{dom}
\DeclareMathOperator{\cod}{cod}

\newcommand{\I}{\ensuremath{\mathcal{I}}}

\theoremstyle{plain}
\newtheorem{theorem}{Theorem}[section]
\newtheorem{corollary}[theorem]{Corollary}
\newtheorem{lemma}[theorem]{Lemma}
\newtheorem{proposition}[theorem]{Proposition}

\theoremstyle{definition}
\newtheorem{definition}[theorem]{Definition}

\newtheorem{example}[theorem]{Example}
\newtheorem{remark}[theorem]{Remark}

\tikzstyle{cdiag}=[matrix of math nodes, row sep=3em, column sep=3em, text height=1.5ex, text depth=0.25ex,inner sep=0.5em]
\tikzstyle{arrow above}=[transform canvas={yshift=0.5ex}]
\tikzstyle{arrow below}=[transform canvas={yshift=-0.5ex}]

\newcommand{\vkcleararrows}{%
\tikzstyle{vkarrow1}=[arrow plain]
\tikzstyle{vkarrow2}=[arrow plain]
\tikzstyle{vkarrow3}=[arrow plain]
\tikzstyle{vkarrow4}=[arrow plain]
\tikzstyle{vkarrow5}=[arrow plain]
\tikzstyle{vkarrow6}=[arrow plain]
\tikzstyle{vkarrow7}=[arrow plain]
\tikzstyle{vkarrow8}=[arrow plain]
\tikzstyle{vkarrow9}=[arrow plain]
\tikzstyle{vkarrow10}=[arrow plain]
\tikzstyle{vkarrow11}=[arrow plain]
\tikzstyle{vkarrow12}=[arrow plain]}
\vkcleararrows

\newcommand{\eq}{\ensuremath{\mathop{\sim}}}
\newcommand{\Eq}{\ensuremath{\mathop{\approx}}}
\newcommand{\B}{\ensuremath{\mathcal{B}}}
\newcommand{\M}{\mathbb{M}}
\DeclareMathOperator{\Dom}{D}

\title{Completely positive projections and biproducts}
\author{Chris Heunen
\institute{Department of Computer Science\\ University of Oxford}
\email{heunen@cs.ox.ac.uk}
\and Aleks Kissinger
\institute{Department of Computer Science\\ University of Oxford}
\email{alek@cs.ox.ac.uk}
\and
Peter Selinger
\institute{Department of Mathematics and Statistics\\ Dalhousie University}
\email{selinger@mathstat.dal.ca}
}

\def\titlerunning{Completely positive projections and biproducts}
\def\authorrunning{Chris Heunen, Aleks Kissinger, and Peter Selinger}

\begin{document}
\maketitle
\begin{abstract}
  The recently introduced CP*--construction unites quantum channels
  and classical systems, subsuming the earlier CPM--construction in
  categorical quantum mechanics. We compare this construction to two
  earlier attempts at solving this problem: freely adding biproducts
  to CPM, and freely splitting idempotents in CPM. The CP*--construction embeds the
  former, and embeds into the latter, but neither embedding is an
  equivalence in general.
\end{abstract}

\section{Introduction}

Two of the authors recently introduced the so-called
\emph{CP*--construction}, turning a 
category $\V$ of abstract state spaces into a category $\CPs[\V]$
of abstract C*-algebras and completely positive
maps~\cite{coeckeheunenkissinger:cpstar}. It accommodates 
both quantum channels and classical systems in a single
category. Moreover, it allows nonstandard models connecting to the
well-studied theory of groupoids. In particular, it subsumes the
earlier CPM--construction, which gives the subcategory $\CPM[\V]$ of
the CP*--construction of abstract matrix
algebras~\cite{selinger:completelypositive}, and adds classical
information to it. 

There have been earlier attempts at uniting quantum channels and
classical systems~\cite{selinger:completelypositive,selinger:daggeridempotents}. This paper
compares the CP*--construction to two of them: freely adding
biproducts to $\CPM[\V]$, and freely splitting the dagger idempotents
of $\CPM[\V]$. These new categories are referred to as
$\CPM[\V]^\oplus$ and $\Split{}[\CPM[\V]]$, respectively. We will prove
that the CP*--construction lies in between these two: there are full
and faithful functors
\[
  \CPM[\V]^\oplus \to \CPs[\V] \to \Split{}[\CPM[\V]].
\]
When $\V$ is the category of finite-dimensional Hilbert spaces, both
outer categories provide ``enough space'' to reason about classical
and quantum data, because any finite-dimensional C*-algebra is a
direct sum of matrix algebras (as in $\CPM[\FHilb]^\oplus$), and a
certain orthogonal subspace of a larger matrix algebra (as in
$\Split{}[\CPM[\FHilb]]$). However, a priori it is unclear whether the
second construction captures too much: it may contain many more
objects than simply mixtures of classical and quantum state
spaces, although none have been discovered so
far~\cite[Remark~4.9]{selinger:daggeridempotents}. On the other  
hand, for $\V\neq\FHilb$, the first construction may not capture
enough: there may be interesting objects that are not just sums of
quantum systems. For this reason, it is interesting to study
$\CPs[\V]$, because the nonstandard models suggest it captures
precisely the right amount of interesting objects.

To be a bit more precise, we will prove that if $\V$ has biproducts,
then $\CPs[\V]$ inherits them. The universal property of the
free biproduct completion then guarantees the first embedding
above. We will show that this embedding is not an equivalence in
general. 

For the second embedding, we construct the associated
dagger idempotent of an object in $\CPs[\V]$, and prove that the notions of
complete positivity in $\CPs[\V]$ and $\Split{}[\CPM[\V]]$
coincide, giving rise to a full and faithful functor. Finally, we will
show that this embedding is not an equivalence in general either.

\subsection*{The CPM and CP*--constructions}

To end this introduction, we very briefly recall the CPM and CP*--constructions. For
more information, we refer to~\cite{coeckeheunenkissinger:cpstar}.
Let $\V$ be a dagger compact category
(see~\cite{abramskycoecke:cqm,selinger:completelypositive}). 
A morphism $f \colon A^* \otimes A \to B^* \otimes B$ is called \emph{completely positive} when
it is of the form
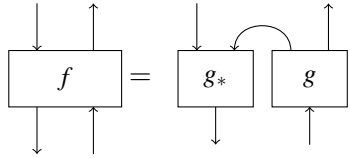
\ctikzfig{cpm}
for some object $C$ and some morphism $g \colon A \to C \otimes B$. The category $\CPM[\V]$ has the same objects and composition as $\V$, but a morphism $A \to B$ in $\CPM[\V]$ is a completely positive morphism $A^* \otimes A \to B^* \otimes B$ in $\V$.
The two main theorems about $\CPM[\V]$ are the following~\cite{selinger:completelypositive}.
First, $\CPM[\V]$ is again a dagger compact category with monoidal structure inherited from $\V$.
Second, $\CPM[\FHilb]$ is the category whose objects are finite-dimensional Hilbert spaces and whose morphisms completely positive maps as usually defined in quantum information theory.

A \emph{dagger Frobenius algebra} is an object $A$ together with
morphisms $\whitemult \colon A 
\otimes A \to A$ and $\whiteunit \colon I \to A$ satisfying:
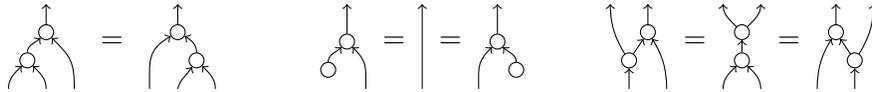
\ctikzfig{fa_axioms}
Any dagger Frobenius algebra defines a cap and a cup satisfying the snake
identities. 
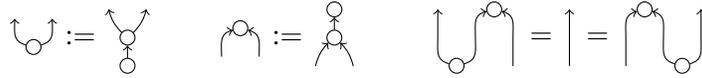
\ctikzfig{cap_cup}
A map $z \colon A \rightarrow A$ is \emph{central} for $\whitemult$
when: 
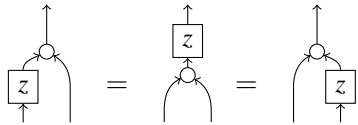
\ctikzfig{central}
A map $g \colon A \to A$ is \emph{positive} if $g=h^\dag \circ h$
for some $h$.  
A dagger Frobenius algebra $(A, \whitemult, \whiteunit)$ is
\emph{normalisable} if it comes with a central, positive isomorphism
$\whitenorm$ with:
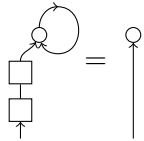
\ctikzfig{normalisable}
A normalisable dagger Frobenius algebra is \emph{normal} when
$\whitenorm=\id[A]$. 

\begin{remark}
  It will often be more convenient to use the \textit{action} and
  \textit{coaction} maps associated with a Frobenius algebra, defined
  as follows: 
  \ctikzfig{action_abbrev}
  Using these maps, we can prove alternative forms of the Frobenius
  and normalisability equations (see Lemmas~2.9 and 2.10 of
  {\cite{coeckeheunenkissinger:cpstar}}). These are: 
  \begin{equation}\label{eqn-lemma-2.9-2.10}
\beginpgfgraphicnamed{sfa_ident}
\begin{tikzpicture}[dotpic]
	\begin{pgfonlayer}{nodelayer}
		\node [style=white dot] (0) at (-1.25, 0.5) {};
		\node [style=white dot] (1) at (-1.25, -0.25) {};
		\node [style=none] (2) at (-1.75, -0.75) {};
		\node [style=none] (3) at (-0.75, -0.75) {};
		\node [style=none] (4) at (-1.75, 1) {};
		\node [style=none] (5) at (-0.75, 1) {};
		\node [style=none] (6) at (2.5, 1) {};
		\node [style=none] (7) at (2.5, -0.75) {};
		\node [style=white dot] (8) at (1.25, -0) {};
		\node [style=none] (9) at (1.25, -0.75) {};
		\node [style=white dot] (10) at (2.5, -0) {};
		\node [style=none] (11) at (1.25, 1) {};
		\node [style=none] (12) at (0, -0) {$=$};
	\end{pgfonlayer}
	\begin{pgfonlayer}{edgelayer}
		\draw [style=diredge, bend right=15] (4.center) to (0);
		\draw [style=diredge, bend right=15] (0) to (5.center);
		\draw [style=diredge, bend right=15] (3.center) to (1);
		\draw [style=diredge, bend right=15] (1) to (2.center);
		\draw [style=diredge] (1) to (0);
		\draw [style=diredge] (11.center) to (8);
		\draw [style=diredge] (8) to (9.center);
		\draw [style=diredge] (7.center) to (10);
		\draw [style=diredge] (10) to (6.center);
		\draw [style=diredge, bend right=60, looseness=1.25] (10) to (8);
	\end{pgfonlayer}
\end{tikzpicture}}
\endpgfgraphicnamed
    \qquad\text{and}\qquad
\beginpgfgraphicnamed{norm_alt}
\begin{tikzpicture}[dotpic]
	\begin{pgfonlayer}{nodelayer}
		\node [style=none] (0) at (1, -1.25) {};
		\node [style=none] (1) at (1, 1.25) {};
		\node [style=none] (2) at (-1.5, 1.5) {};
		\node [style=none] (3) at (-1.5, -1.5) {};
		\node [style=none] (4) at (0, -0) {$=$};
		\node [style=white dot] (5) at (-1.5, 0.75) {};
		\node [style=white dot] (6) at (-1.5, -0.75) {};
		\node [style=white norm] (7) at (-2, -0) {};
		\node [style=white norm] (8) at (-1, -0) {};
	\end{pgfonlayer}
	\begin{pgfonlayer}{edgelayer}
		\draw [style=diredge] (0.center) to (1.center);
		\draw [style=diredge] (5) to (2.center);
		\draw [style=diredge] (3.center) to (6);
		\draw [style=diredge, in=-90, out=30] (6) to (8);
		\draw [style=diredge, in=-30, out=90] (8) to (5);
		\draw [style=diredge, in=90, out=-150] (5) to (7);
		\draw [style=diredge, in=150, out=-90] (7) to (6);
	\end{pgfonlayer}
\end{tikzpicture}}
\endpgfgraphicnamed
  \end{equation}
\end{remark}

Finally, the category $\CPs[\V]$ is defined as follows. Objects are
normalisable dagger Frobenius algebras in $\V$. Morphisms
$(A,\whitemult,\whiteunit,\whitenorm) \to
(B,\blackmult,\blackunit,\blacknorm)$ are morphisms $f \colon A \to B$ in
$\V$ such that $\blackaction\circ f\circ\whitecoaction$ is \emph{completely
positive}, i.e., such that
\begin{equation}\label{eqn-f-cp}
\beginpgfgraphicnamed{cpstar_condition}
\begin{tikzpicture}[dotpic]
	\begin{pgfonlayer}{nodelayer}
		\node [style=none] (0) at (0, -0) {$=$};
		\node [style=square box] (1) at (-2, -0) {$f$};
		\node [style=black dot] (2) at (-2, 1) {};
		\node [style=none] (3) at (-1.25, 2) {};
		\node [style=none] (4) at (-2.75, 2) {};
		\node [style=white dot] (5) at (-2, -1) {};
		\node [style=none] (6) at (-1.25, -2) {};
		\node [style=none] (7) at (-2.75, -2) {};
		\node [style=square box, minimum width=1 cm, minimum height=0.75 cm] (8) at (2, -0) {$g_*$};
		\node [style=none] (9) at (4.5, -0.75) {};
		\node [style=none] (10) at (2, -0.75) {};
		\node [style=none] (11) at (2, -1.75) {};
		\node [style=none] (12) at (4.5, -1.75) {};
		\node [style=none] (13) at (5, 0.75) {};
		\node [style=none] (14) at (1.5, 0.75) {};
		\node [style=none] (15) at (5, 2) {};
		\node [style=none] (16) at (1.5, 2) {};
		\node [style=square box, minimum width=1 cm, minimum height=0.75 cm] (17) at (4.5, -0) {$g$};
		\node [style=none] (18) at (2.5, 0.75) {};
		\node [style=none] (19) at (4, 0.75) {};
	\end{pgfonlayer}
	\begin{pgfonlayer}{edgelayer}
		\draw [style=diredge] (1) to (2);
		\draw [style=diredge, bend right] (4.center) to (2);
		\draw [style=diredge, bend right] (2) to (3.center);
		\draw [style=diredge, bend right] (5) to (7.center);
		\draw [style=diredge, bend right] (6.center) to (5);
		\draw [style=diredge] (5) to (1);
		\draw [style=diredge] (12.center) to (9.center);
		\draw [style=diredge] (13.center) to (15.center);
		\draw [style=diredge] (10.center) to (11.center);
		\draw [style=diredge] (16.center) to (14.center);
		\draw [style=diredge, in=90, out=90, looseness=1.50] (19.center) to (18.center);
	\end{pgfonlayer}
\end{tikzpicture}}
\endpgfgraphicnamed
\end{equation}
for some object $X$ and morphism $g \colon A \to X \otimes B$ in $\V$.
This category inherits the dagger compact structure from $\V$~\cite[Theorem~3.4]{coeckeheunenkissinger:cpstar}.
We write $\CPsn[\V]$ for the full subcategory whose objects are normal
dagger Frobenius algebras in $\V$.  
Recall that $\CPM[\V]$ has the same objects as $\V$, and morphisms $A
\to B$ are completely positive maps $A^* \otimes A \to B^* \otimes B$~\cite[Definition~4.18]{selinger:completelypositive}.

\begin{lemma}\label{lem:normal}
  Any normalisable dagger Frobenius algebra in $\V$ is 
  isomorphic in $\CPs[\V]$ to a normal one.
\end{lemma}
\begin{proof}
  For an object $(A,\whitemult,\whiteunit,\whitenorm)$ of
  $\CPs[\V]$, define $\blackmult = \whitenorm \circ \whitemult$ and
  $\blackunit = \whitenorm^{-1} \circ 
  \whiteunit$. It follows from centrality and self-adjointness of
  $\whitenorm$ that $(A,\blackmult,\blackunit)$ is a dagger Frobenius
  algebra in $\V$. Moreover $\blackcap = \whitecap$, and so
  $(A,\blackmult,\blackunit)$ is normal.
  Finally, $\id[A]$ is a well-defined morphism
  from $(A,\whitemult,\whiteunit,\whitenorm)$ to
  $(A,\blackmult,\blackunit,\id[A])$ in $\CPs[\V]$: if $\whitenorm = 
  \graynorm^\dag \circ \graynorm$, then 
  \ctikzfig{normal}
  where the second equality follows from (\ref{eqn-lemma-2.9-2.10}).
  The morphism $\id[A]$ is a unitary isomorphism.
\end{proof}

\begin{remark}\label{rem:CPsn}
  The previous lemma shows that $\CPs[\V]$ and $\CPsn[\V]$ are dagger
  equivalent. Therefore, all properties that we can prove in one
  automatically transfer to the other, as long as the properties in
  question are invariant under dagger equivalence. All results in this
  paper are of this kind, and hence we lose no generality by assuming
  that all normalisable dagger Frobenius algebras have been
  normalised.
\end{remark}

\section{Splitting idempotents}

This section exhibits a canonical full and faithful functor from
$\CPs[\V]$ into $\Splitd[\CPM[\V]]$. This functor is not an
equivalence for $\V=\Rel$. It is not known whether it is an equivalence for
$\V=\FHilb$.  However, we characterise its image, showing that the image
is equivalent to the full subcategory of $\Splitd[\CPM[\FHilb]]$
consisting of unital dagger idempotents.

\begin{definition}
  Let $\mathcal{I}$ be a class of pairs $(X,p)$, where $X$ is an object
  of $\V$, and $p \colon X \to X$ is a morphism in $\V$ satisfying
  $p^\dag=p=p \circ p$, called a \emph{dagger idempotent} or \emph{projection}.  
  The category $\Split{\I}[\V]$ has $\I$ as objects. Morphisms $(X,p) \to (Y,q)$ in
  $\Split{\I}[\V]$ are morphisms $f \colon X \to Y$ in $\V$ satisfying $f
  = q \circ f \circ p$.
\end{definition}

If $\I$ is closed under tensor, then $\Split{\I}[\V]$ is dagger compact~\cite[Proposition~3.16]{selinger:daggeridempotents}.
When $\I$ is the class of all dagger idempotents in $\V$, we also
write $\Splitd[\V]$ instead of $\Split{\I}[\V]$.

\begin{lemma}\label{lem:CPsintoSplit}
  Let $\V$ be any dagger compact category. Then there is a canonical functor $F
  \colon \CPs[\V] \to \Splitd[\CPM[\V]]$
  acting as $F(A,\whitemult,\whiteunit,\whitenorm) = (A,\whiteaction \circ
  \whitenorm \circ \whitenorm \circ \whitecoaction)$ on
  objects, and as $F(f)= \blackaction \circ \blacknorm \circ f \circ \whitenorm
  \circ \whitecoaction$ on morphisms. It is full, faithful,
  and strongly dagger symmetric monoidal. 
\end{lemma}
\begin{proof}
  First, note that $p=F(A,\whitemult,\whiteunit,\whitenorm)$ is a
  well-defined object of $\Splitd[\CPM[\V]]$: clearly $p=\whiteaction
  \circ \whitenorm \circ \whitenorm \circ \whitecoaction=p^\dag$;
  also, it follows from (\ref{eqn-lemma-2.9-2.10}) that $p \circ p =
  p$ and that $p$ is completely positive. 
  Also, $F(f)$ is a well-defined morphism in
  $\CPM[\V]$ by (\ref{eqn-f-cp}). 
  By (\ref{eqn-lemma-2.9-2.10}), $F(f)$ is in fact a
  well-defined morphism in $\Splitd[\CPM[\V]]$. 
  Next, $F$ is faithful because $f = \blacknorm \circ \blackcoaction \circ F(f)
  \circ \whiteaction \circ \whitenorm$. 
  To show that $F$ is full, note that an arbitrary morphism
  $h \colon A^* \otimes A \to B^* \otimes B$ in $\V$ is a well-defined
  morphism in $\CPs[\V]$ if and only if it is a well-defined morphism
  in $\Splitd[\CPM[\V]]$:
  \ctikzfig{cpstar_into_split}  
  Both $\CPs[\V]$ and $\Splitd[\CPM[\V]]$ inherit composition,
  identities, and daggers from $\V$, so $F$ is a full and faithful
  functor preserving daggers. Similarly, the
  symmetric monoidal structure in both $\CPs[\V]$ and
  $\Splitd[\CPM[\V]]$ is defined in terms of that of $\V$,
  making $F$ strongly symmetric monoidal. 
\end{proof}

It stands to reason that $F$ might become an equivalence by
restricting the class $\I$. For example, the following lemma shows
that we should at least restrict to splitting unital
projections. A completely positive map $f \colon A^*\otimes A\to
B^*\otimes B$ is \emph{unital} when:
\[ %
\beginpgfgraphicnamed{unital}
\begin{tikzpicture}[dotpic]
	\begin{pgfonlayer}{nodelayer}
		\node [style=none] (0) at (-2, 0.75) {};
		\node [style=none] (1) at (-1, 0.75) {};
		\node [style=none] (2) at (0.25, -0.25) {$=$};
		\node [style=none] (3) at (1, 0.25) {};
		\node [style=none] (4) at (2, 0.25) {};
		\node [style=none] (5) at (1.5, -0.5) {};
		\node [style=none] (6) at (-2, -1) {};
		\node [style=none] (7) at (-1, -1) {};
		\node [style=square box, minimum width=1 cm] (8) at (-1.5, -0.25) {$f$};
		\node [style=none] (9) at (-2, 0.35) {};
		\node [style=none] (10) at (-1, 0.35) {};
		\node [style=none] (11) at (-2, -0.85) {};
		\node [style=none] (12) at (-1, -0.85) {};
	\end{pgfonlayer}
	\begin{pgfonlayer}{edgelayer}
		\draw [style=none, in=180, out=-90] (3.center) to (5.center);
		\draw [style=diredge, in=-90, out=0] (5.center) to (4.center);
		\draw [style=diredge] (0.center) to (9.center);
		\draw [style=diredge] (10.center) to (1.center);
		\draw [style=diredge] (7.center) to (12.center);
		\draw [style=none] (11.center) to (6.center);
		\draw [style=none, bend right=90, looseness=1.50] (6.center) to (7.center);
	\end{pgfonlayer}
\end{tikzpicture}}
\endpgfgraphicnamed
\]

\begin{lemma}
  The functor $F$ from Lemma~\ref{lem:CPsintoSplit} lands in
  $\Split{\I}[\CPM[\V]]$, where $\I$ consists of the unital dagger
  idempotents.
\end{lemma}
\begin{proof}
  Let $(A,\whitemult,\whiteunit,\whitenorm)$ be an object of
  $\CPs[\V]$. Then $F(A)$ is unital because:
  \[ %
\beginpgfgraphicnamed{split_unital}
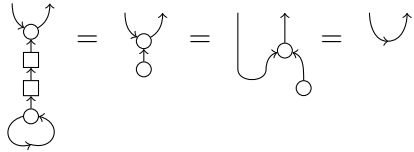
\begin{tikzpicture}[dotpic]
	\begin{pgfonlayer}{nodelayer}
		\node [style=none] (0) at (0, 1.5) {$=$};
		\node [style=white dot] (1) at (-1.5, -0.5) {};
		\node [style=white dot] (2) at (-1.5, 1.75) {};
		\node [style=white norm] (3) at (-1.5, 0.25) {};
		\node [style=white norm] (4) at (-1.5, 1) {};
		\node [style=none] (5) at (-2, 2.5) {};
		\node [style=none] (6) at (-1, 2.5) {};
		\node [style=none] (7) at (-1.5, -1.25) {};
		\node [style=none] (8) at (1, 2.25) {};
		\node [style=none] (9) at (2, 2.25) {};
		\node [style=white dot] (10) at (1.5, 1.5) {};
		\node [style=white dot] (11) at (1.5, 0.75) {};
		\node [style=none] (12) at (3, 1.5) {$=$};
		\node [style=none] (13) at (7.5, 2.25) {};
		\node [style=none] (14) at (6.5, 1.5) {$=$};
		\node [style=none] (15) at (8.5, 2.25) {};
		\node [style=none] (16) at (8, 1.5) {};
		\node [style=none] (17) at (4, 0.75) {};
		\node [style=white dot] (18) at (5.25, 1.25) {};
		\node [style=white dot] (19) at (5.75, 0.25) {};
		\node [style=none] (20) at (4.75, 0.75) {};
		\node [style=none] (21) at (5.25, 2.25) {};
		\node [style=none] (22) at (4, 2.25) {};
	\end{pgfonlayer}
	\begin{pgfonlayer}{edgelayer}
		\draw [style=diredge] (1) to (3);
		\draw [style=diredge] (4) to (2);
		\draw [style=diredge] (3) to (4);
		\draw [style=diredge, in=-90, out=30] (2) to (6.center);
		\draw [style=diredge, in=150, out=-90] (5.center) to (2);
		\draw [style=diredge, bend right=90, looseness=2.25] (1) to (7.center);
		\draw [style=diredge, bend right=90, looseness=2.25] (7.center) to (1);
		\draw [style=diredge, in=-90, out=30] (10) to (9.center);
		\draw [style=diredge, in=150, out=-90] (8.center) to (10);
		\draw [style=diredge] (11) to (10);
		\draw [style=diredge, in=180, out=-90] (13.center) to (16.center);
		\draw [style=diredge, in=-90, out=0] (16.center) to (15.center);
		\draw [style=diredge, in=-15, out=90] (19) to (18);
		\draw [style=diredge] (18) to (21.center);
		\draw (22.center) to (17.center);
		\draw [style=diredge, in=-165, out=90] (20.center) to (18);
		\draw [in=-90, out=-90, looseness=1.25] (17.center) to (20.center);
	\end{pgfonlayer}
\end{tikzpicture}}
\endpgfgraphicnamed\qedhere
  \]
\end{proof}

From now on, we will fix $\I$ to be the class of unital dagger idempotents.

\subsection*{Hilbert spaces}

When $\V=\FHilb$, the objects of $\CPs[\FHilb]$ are precisely
(concrete) C*-algebras, and the morphisms are 
completely positive maps in the usual sense of C*-algebras; see
{\cite{coeckeheunenkissinger:cpstar}}.
The unital completely positive maps $p \in
\I$ are precisely the physically realisable projections. We will prove
that $F$ is then an equivalence, by employing a classic theorem
by Choi and Effros. It is well-known that the image $f(A)$ of a
*-homomorphism $f \colon A \to B$ is a C*-subalgebra of $B$. The
Choi-Effros theorem shows that the image $p(A)$ of a completely positive unital
projection $p \colon A \to A$ is a C*-algebra in its own
right. In general, it need no longer be a C*-subalgebra; it can have a
different multiplication. The following proposition and its proof make
precise what we need. 
Write $\M_n(A)$ for the C*-algebra of $n$-by-$n$ matrices with
entries in $A$, and simply $\M_n$ for
$\M_n(\mathbb{C})$. The assignment $A \mapsto \M_n(A)$ is functorial 
on the category of C*-algebras and *-homomorphisms. The category $\CPM[\FHilb]$ can
be identified with the full subcategory of $\CPs[\FHilb]$ consisting
of the matrix algebras $\M_n$.

\begin{proposition}\label{prop:G}
  There is a functor $G \colon \Split{\I}[\CPM[\FHilb]] \to \CPs[\FHilb]$
  that sends an object $(\M_m,p)$ to its range $p(\M_m)$, and 
  a morphism $f \colon (\M_m,p) \to (\M_n,q)$ to its underlying
  function $f \colon \M_m \to \M_n$.
\end{proposition}
\begin{proof}
  Because $p$ is unital, it is certainly contractive, because it has
  operator norm $\|p\| =\|p(1)\| = \|1\| = 1$. Therefore, a classic
  theorem by Choi and Effros applies, showing that $p(\M_m)$ is
  a well-defined C*-algebra under the product $(a,b) \mapsto
  p(ab)$~\cite[Theorem~3.1]{choieffros:injectivity} (see
  also~\cite[Section~2.2]{stormer:positive}). 
  Hence $G$ is well-defined on objects. Because dagger idempotents
  dagger split in $\FHilb$, this can be denoted graphically as
  \[
  G(\M_m \stackrel{p}{\to} \M_m) = (p(\M_m), %
\beginpgfgraphicnamed{choieffros}
\begin{tikzpicture}[dotpic]
	\begin{pgfonlayer}{nodelayer}
		\node [style=none] (0) at (-1, -0.75) {};
		\node [style=none] (1) at (0.5, -0.75) {};
		\node [style=none] (2) at (-0.25, 0.25) {};
		\node [style=none] (3) at (-0.5, -0.75) {};
		\node [style=none] (4) at (1, -0.75) {};
		\node [style=none] (5) at (0.25, 0.25) {};
		\node [style=none] (6) at (0, 0.65) {};
		\node [style=none] (7) at (0, 1.15) {};
		\node [style=none] (8) at (-0.75, -1.15) {};
		\node [style=none] (9) at (-0.75, -1.75) {};
		\node [style=none] (10) at (0.75, -1.15) {};
		\node [style=none] (11) at (0.75, -1.75) {};
		\node [style=wide white point] (12) at (-0.75, -0.9) {};
		\node [style=wide white point] (13) at (0.75, -0.9) {};
		\node [style=wide white copoint] (14) at (0, 0.4) {};
	\end{pgfonlayer}
	\begin{pgfonlayer}{edgelayer}
		\draw [style=diredge] (6.center) to (7.center);
		\draw [style=diredge, in=90, out=-90, looseness=0.75] (2.center) to (0.center);
		\draw [style=diredge, in=-90, out=90] (4.center) to (5.center);
		\draw [style=diredge, bend left=90] (3.center) to (1.center);
		\draw [style=diredge] (9.center) to (8.center);
		\draw [style=diredge] (11.center) to (10.center);
	\end{pgfonlayer}
\end{tikzpicture}}
\endpgfgraphicnamed, \whitenorm),
  \]
  where $\splitcotriangle \colon (\C^m)^* \otimes \C^m \to p(\M_m)$ is (a dagger
  splitting of) $p$ with inclusion $\splittriangle \colon p(\M_m) \to
  (\C^m)^* \otimes \C^m$, and $\whitenorm$ is the unique normaliser. That
  is, we have $\id[p(\M_m)] = p = \splitcotriangle \circ
  \splittriangle \colon p(\M_m) \to p(\M_m)$.

  Choi and Effros~\cite[Theorem~3.1]{choieffros:injectivity} also
  study how positivity in $p(\M_m)$ and $\M_m$ is related.
  Specifically, they prove that $\M_k(p(\M_m))^+ = \M_k(\M_m)^+ \cap
  \M_k(p(\M_m))$, where $A^+$ denotes the positive cone of a C*-algebra $A$.
  To see that $G$ is well-defined on morphisms, let $f \colon
  (\M_m,p) \to (\M_n,q)$ be a morphism in
  $\Split{\I}[\CPM[\FHilb]]$. 
  Then $G(f)$ is a well-defined completely positive map $p(\M_m) \to
  q(\M_n)$ precisely when $x \in \M_k(p(\M_m))^+$ implies $\M_k f(x)
  \in \M_k(q(\M_n))^+$ for all $k$. But this is indeed true because $f
  \colon \M_m \to \M_n$ is a completely positive map satisfying $f=q
  \circ f \circ p$.  Finally, $G$ is clearly functorial.
\end{proof}

\begin{theorem}
  The functors $F$ and $G$ implement an equivalence between the categories
  $\CPs[\FHilb]$ and $\Split{\I}[\CPM[\FHilb]]$.
\end{theorem}
\begin{proof}
  Let $p \colon \M_m \to \M_m$ be a completely positive unital projection.
  We will show that $F(G(p)) \cong p$. This will establish that $F$ is
  essentially surjective on objects. Since it is also full and
  faithful, it follows that $F$ is an equivalence. Using
  the graphical notation from the proof of Proposition~\ref{prop:G}, define 
  \begin{align*}
    g & = %
\beginpgfgraphicnamed{split_hilb_g}
\begin{tikzpicture}[dotpic]
	\begin{pgfonlayer}{nodelayer}
		\node [style=none] (0) at (-1, -0.75) {};
		\node [style=none] (1) at (0.5, -0.75) {};
		\node [style=none] (2) at (-0.25, 0.25) {};
		\node [style=none] (3) at (-0.5, -0.75) {};
		\node [style=none] (4) at (1, -0.75) {};
		\node [style=none] (5) at (0.25, 0.25) {};
		\node [style=none] (6) at (0, 0.65) {};
		\node [style=none] (7) at (0, 1) {};
		\node [style=none] (8) at (-0.75, -1.15) {};
		\node [style=none] (9) at (-0.75, -1.5) {};
		\node [style=none] (10) at (0.75, -1.15) {};
		\node [style=none] (11) at (0.75, -1.5) {};
		\node [style=wide white point] (12) at (-0.75, -0.9) {};
		\node [style=wide white point] (13) at (0.75, -0.9) {};
		\node [style=wide white copoint] (14) at (0, 0.4) {};
		\node [style=wide white point] (15) at (0, 1.25) {};
		\node [style=none] (16) at (0.25, 1.4) {};
		\node [style=none] (17) at (-0.25, 1.4) {};
		\node [style=none] (18) at (0.25, 1.75) {};
		\node [style=none] (19) at (-0.25, 1.75) {};
	\end{pgfonlayer}
	\begin{pgfonlayer}{edgelayer}
		\draw [style=diredge] (6.center) to (7.center);
		\draw [style=diredge, in=90, out=-90, looseness=0.75] (2.center) to (0.center);
		\draw [style=diredge, in=-90, out=90] (4.center) to (5.center);
		\draw [style=diredge, bend left=90] (3.center) to (1.center);
		\draw [style=diredge] (8.center) to (9.center);
		\draw [style=diredge] (11.center) to (10.center);
		\draw [style=diredge] (19.center) to (17.center);
		\draw [style=diredge] (16.center) to (18.center);
	\end{pgfonlayer}
\end{tikzpicture}}
\endpgfgraphicnamed \colon p(\M_m)^* \otimes p(\M_m) \to
    \M_m, \\[1mm]
    f & = %
\beginpgfgraphicnamed{split_hilb_f}
\begin{tikzpicture}[dotpic]
	\begin{pgfonlayer}{nodelayer}
		\node [style=none] (0) at (-1, -0.5) {};
		\node [style=none] (1) at (0.5, -0.5) {};
		\node [style=none] (2) at (-0.25, -1.5) {};
		\node [style=none] (3) at (-0.5, -0.5) {};
		\node [style=none] (4) at (1, -0.5) {};
		\node [style=none] (5) at (0.25, -1.5) {};
		\node [style=none] (6) at (-0.75, -0.1) {};
		\node [style=none] (7) at (0.75, -0.1) {};
		\node [style=wide white copoint] (8) at (-0.75, -0.35) {};
		\node [style=wide white copoint] (9) at (0.75, -0.35) {};
		\node [style=none] (10) at (-0.75, 1) {};
		\node [style=none] (11) at (0.75, 1) {};
		\node [style=white norm] (12) at (-0.75, 0.5) {};
		\node [style=white norm] (13) at (0.75, 0.5) {};
	\end{pgfonlayer}
	\begin{pgfonlayer}{edgelayer}
		\draw [style=diredge, in=90, out=-90, looseness=0.75] (0.center) to (2.center);
		\draw [style=diredge, in=-90, out=90] (5.center) to (4.center);
		\draw [style=diredge, bend left=90] (1.center) to (3.center);
		\draw [style=diredge] (10.center) to (12);
		\draw [style=diredge] (12) to (6.center);
		\draw [style=diredge] (13) to (11.center);
		\draw [style=diredge] (7.center) to (13);
	\end{pgfonlayer}
\end{tikzpicture}}
\endpgfgraphicnamed \colon \M_m \to p(\M_m)^* \otimes p(\M_m).
  \end{align*}
  Then $f$ is in Kraus form, and hence completely positive, by
  construction. Similarly, $g$ is the composition of
  $p = \splittriangle\circ\splitcotriangle$, which is completely
  positive by assumption, and another map that is completely positive
  by construction. Hence
  $f$ and $g$ are well-defined morphisms in $\CPM[\FHilb]$. Moreover,
  by (\ref{eqn-lemma-2.9-2.10}),
  $g \circ f = \splittriangle \circ \splitcotriangle = p \colon \M_m
  \to \M_m$. Also,
  \ctikzfig{split_hilb_fg}
  Therefore $f \circ g = F(G(p)) \colon p(\M_m)^* \otimes p(\M_m) \to
  p(\M_m)^* \otimes p(\M_m)$. It follows that $f=F(G(p)) \circ f \circ
  p$ and $g = p \circ g \circ F(G(p))$, making $f$ and $g$ into
  well-defined morphisms of $\Split{\I}[\CPM[\FHilb]]$. In fact, this
  shows that $f$ and $g$ implement an isomorphism in that category,
  establishing $F(G(p)) \cong p$.

  It now follows that if $A \in \CPs[\FHilb]$, then $F(G(F(A))) \cong
  F(A)$, and because $F$ is full and faithful, hence $G(F(A)) \cong
  A$. It is easy to see that this isomorphism, as well as $F(G(p))
  \cong p$, is natural. Thus $F$ and $G$ form an equivalence.
\end{proof}

\begin{remark}
  To motivate the need to restrict to the class of unital projections
  $\I$, let us show that not every 
  object in $\Splitd[\CPM[\FHilb]]$ is unital.  
  We give a counterexample of a completely positive projection that is
  not even contractive.\footnote{We thank Erling St{\o}rmer for discussions on
  this subject.} Take $A=\M_n$, and let $a \in A$ satisfy $a
  \geq 0$, $\|a\|>1$, and $\Tr(a)=\Tr(a^2)$. For example, we could
  pick $n=2$ and 
  \[
    a=\begin{pmatrix}
      \frac{1}{2}+\frac{1}{2}\sqrt{2} & 0 \\ 0 &
      \frac{1}{2} \end{pmatrix}.
  \]
  We will define $p$ as the orthogonal projection onto the
  one-dimensional subspace spanned by a suitable density matrix
  $\rho$. Precisely, define $\rho \in A$, $f \colon A \to \C$, and $p
  \colon A \to A$ by 
  \[
    \rho=\frac{a}{\Tr(a)}, \qquad 
    f(x)=\Tr(\rho x), \qquad
    p(x)= f(x)a.
  \]
  Then $\rho\geq 0$ and $\Tr(\rho)=1$, so $\rho$ is a density matrix.
  The adjoint of $p$ with respect to the trace inner product
  $\inprod{x}{y}=\Tr(x^\dag y)$ is $p^\dag(x) = \rho \Tr(ax)$:
  \[
    \Tr(p(x) y) = \Tr(\rho x) \Tr(ay) = \Tr(x p^\dag(y)).
  \]
  Hence $p$ is self-adjoint:
  \[
    p^\dag(x) = \rho \Tr(ax) = \frac{\Tr(ax)a}{\Tr(a)} = \Tr(\rho x)a = p(x).
  \]
  It is also idempotent, because $f(a)=\Tr(\rho a)=\frac{\Tr(a^2)}{\Tr(a)}=1$:
  \[
    p^2(x) = p(\Tr(\rho x) a) = \Tr(\rho x) p(a) = \Tr(\rho x)
    \Tr(\rho a) a = \Tr(\rho x) a = p(x).
  \]
  Thus $p$ is a well-defined object of $\Splitd[\CPM[\FHilb]]$.
  But by the Russo-Dye theorem~\cite[Theorem~1.3.3]{stormer:positive},
  the operator norm of $p$ is
  \[
    \|p\| = \|p(1)\| = \|\Tr(\rho) a\| = \|a\| > 1.
  \]
  Hence $p$ is not contractive, and in particular, not unital. We
  leave open the question whether every object of
  $\Splitd[\CPM[\FHilb]]$ is isomorphic to a unital one. 
\end{remark}

\subsection*{Sets and relations}

Now consider $\V=\Rel$, the category of sets and
relations. We will show that in this case, the canonical functor $F$ is
not an equivalence, even when restricting to the class of unital
projections $\I$, and in fact that there
can be no dagger equivalence at all. 

Recall from~\cite{coeckeheunenkissinger:cpstar} that the category
$\CPs[\Rel]$ has small groupoids $\cat{G}$ as objects; morphisms
$\cat{G} \to \cat{H}$ are relations $R \colon \Mor(\cat{G}) \to
\Mor(\cat{H})$ satisfying   
\begin{align}\label{eq:cprel}
  \text{ if } gRh, 
  \text{ then } g^{-1} R h^{-1} 
  \text{ and } \id[\dom(g)] R \id[\dom(h)].
\end{align}
Notice that $\CPs[\Rel]=\CPsn[\Rel]$ because the only positive
isomorphisms in $\Rel$ are identities.  

We say that two dagger categories $\cat{C}$ and $\cat{D}$ are
\emph{dagger equivalent} when there exist dagger functors $F \colon
\cat{C} \to \cat{D}$ and $G \colon \cat{D} \to \cat{C}$ and natural unitary
isomorphisms $G\circ F \cong \id[\cat{C}]$ and $F \circ G \cong
\id[\cat{D}]$. 

\begin{lemma}
 If all dagger idempotents dagger split in a dagger category
 $\cat{C}$, then they do so in any dagger equivalent category $\cat{D}$.
\end{lemma}
\begin{proof}
  Let $p \colon X \to X$ be a dagger idempotent in $\cat{D}$.
  Then $G(p)$ is a dagger idempotent in $\cat{C}$, and hence dagger
  splits; say $f \colon G(X) \to Y$ satisfies $G(p) = f^\dag \circ
  f$ and $f \circ f^\dag = \id[Y]$. Let $u$ be the unitary isomorphism
  $X \to F(G(X))$, and set $g = F(f) \circ u \colon X \to F(Y)$. Then
  $p = g^\dag \circ g$ and $g \circ g^\dag = F(\id[Y]) = \id[F(Y)]$.
\end{proof}

\begin{theorem}\label{thm:FRelnotequivalence}
  The categories $\CPs[\Rel]$ and $\Splitd[\CPM[\Rel]]$ cannot be
  dagger equivalent.
\end{theorem}
\begin{proof}
  By the previous lemma it suffices to exhibit a dagger idempotent in
  $\CPs[\Rel]$ that does not dagger split. Let $\cat{G}$ be the
  connected groupoid with 3 objects and 9 morphisms:
  \[\cat{G} = 
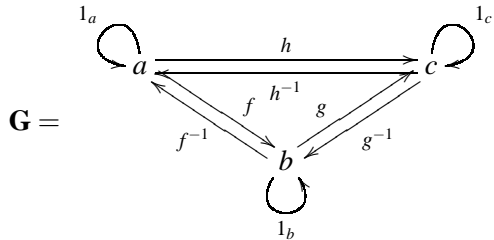
\begin{aligned}\xymatrix@C+4ex{
    a \ar@(u,l)_-{\id[a]} \ar@<.5ex>^(.66){f}[dr] \ar@<.5ex>^-{h}[rr]
    && c \ar@(u,r)^-{\id[c]} \ar@<.5ex>^-{h^{-1}}[ll] \ar@<.5ex>^-{g^{-1}}[dl]
    \\ & b \ar@(dl,dr)_-{\id[b]} \ar@<.5ex>^-{f^{-1}}[ul] \ar@<.5ex>^(.33){g}[ur]
  }\end{aligned}\]
  Write $G$ for the set of morphisms of $\cat{G}$, and define $R = \{
  (x,x) \mid x \in G \setminus \{h,h^{-1}\} \} \subseteq G \times G$. Then $R$
  satisfies~\eqref{eq:cprel}, and hence is a well-defined morphism in
  $\CPs[\Rel]$. Moreover, it is a dagger idempotent. Suppose that $R$
  dagger splits via some $S \subseteq G \times H$; concretely, this means
  $H$ is the morphism set of some groupoid $\cat{H}$, and $S$ satisfies
  equation~\eqref{eq:cprel}, $R=S^\dag \circ S$, and $S \circ S^\dag =
  \id[H]$.
  It follows from $R=S^\dag \circ S$ that $x$ is related by $S$ to
  some element of $H$ if and only if $x$ is neither $h$ nor
  $h^{-1}$. It also follows from $S^\dag \circ S = R \subseteq \id[G]$
  that $xSy$ and $x'Sy$ imply $x=x'$. Hence $S^\dag$ is
  single-valued. Furthermore, it follows from $S \circ S^\dag =
  \id[H]$ that any $y\in H$ relates to some $x \in G$ by $S^\dag$, and
  that $xSy$ and $xSy'$ imply $y=y'$. Thus $S$ is (the graph of)
  a bijection $\{\id[a],\id[b],\id[c],f,f^{-1},g,g^{-1}\} \to H$, and
  $S^\dag$ is (the graph of) its inverse. Hence $\cat{H}$ must have 7
  morphisms. 

  If $S(f)$ were an endomorphism, then $S(\id[a]) = \id[\dom(f)] =
  S(\id[b])$ by~\eqref{eq:cprel}, contradicting injectivity of
  $S$. Similarly, $S(g)$ cannot be an endomorphism. So we may assume
  that $\dom(S(\id[a])) \stackrel{S(f)}{\to} \dom(S(\id[b]))
  \stackrel{S(g)}{\to} \dom(S(\id[c]))$ with $S(\id[a]) \neq
  S(\id[b])$. But because $\cat{H}$ is a groupoid, there must exist a
  morphism $\dom(S(\id[a])) \to \dom(S(\id[b]))$, which contradicts
  the fact that $\cat{H}$ can only have 7 morphisms.
\end{proof}

\begin{corollary}
  The functor $F \colon \CPs[\Rel] \to \Splitd[\CPM[\Rel]]$ is not an equivalence.
\end{corollary}
\begin{proof}
  Suppose $G \colon \Splitd[\CPM[\Rel]] \to \CPs[\Rel]$ and $F$ form
  an equivalence with natural isomorphism $\eta_X \colon F(G(X)) \to X$.
  Let $g \colon X \to Y \in \Splitd[\CPM[\Rel]]$.
  Because every isomorphism in $\Rel$ is unitary, and $F$ preserves
  daggers, 
  \[
    F(G(g^\dag)) 
    = \eta_X^\dag \circ g^\dag \circ \eta_Y 
    = (\eta_Y^\dag \circ g \circ \eta_X)^\dag
    = F(G(g))^\dag
    = F(G(g)^\dag).
  \]
  Since $F$ is faithful, $G$ must also preserve daggers. So $F$
  and $G$ in fact form a dagger equivalence. But that contradicts the
  previous theorem.
\end{proof}

To show that $F$ is not
an equivalence even when we restrict to splitting just unital
projections $\I$, we need to analyse the isomorphisms in 
$\Splitd[\CPM[\Rel]]$ further. This is what the rest of this section does.

\begin{lemma}
  Dagger idempotents in $\Rel$ are precisely partial equivalence relations.
\end{lemma}
\begin{proof}
  Clearly $R^\dag = R$ if and only if $R$ is symmetric. Also, $R^2 \subseteq R$ if and only if $R$ is transitive. We will prove that if $R$ is symmetric, then also $R \subseteq R^2$. Suppose $xRz$. Then also $zRx$, and so $xRx$. Hence $xRyRz$ for $y=x$.
\end{proof}

It follows that the category $\Splitd[\Rel]$ has pairs $(X,\eq)$ as
objects, where $X$ is a set, and $\eq$ is a partial equivalence
relation on $X$; morphisms $(X,\eq) \to (Y,\Eq)$ are relations $R
\colon X \to Y$ satisfying $R= \Eq \circ R \circ \eq$.
For a partial equivalence relation $\eq$ on $X$, we write $\Dom(\eq) = \{ x \in X \mid x \sim x \}$.

\begin{lemma}
  Dagger idempotents in $\Rel$ dagger split.
\end{lemma}
\begin{proof}
  Let $\eq$ be a partial equivalence relation on $X$. Define a
  splitting relation $R \colon \Dom(\eq)/\eq \to X$ by $R=
  \{([x]_{\eq},x) \mid x \in \Dom(\eq)\}$. Then
  \begin{align*}
    R^\dag \circ R
    & = \{([x]_{\eq},[z]_{\eq}) \mid x,z \in \Dom(\eq), \exists y \in X \colon x \sim y \sim z\} \\
    & = \{([x]_{\eq},[z]_{\eq}) \mid x,z \in \Dom(\eq), x \sim z\} \\
    & = \{([x]_{\eq},[x]_{\eq}) \mid x \in \Dom(\eq) \} \\
    & = \id[\Dom(\eq)/\eq],
  \end{align*}
  and $R \circ R^\dag = \{(x,z) \mid x,z \in X, \exists y \in \Dom(\eq)/\eq \colon x \sim y \sim z\} = \eq$.
\end{proof}

Recall that the category $\CPM[\Rel]$ has sets $X$ as objects; morphisms $X \to Y$
are relations $R \colon X \times X \to Y \times Y$ satisfying 
\begin{align}\label{eq:cpmrel} 
  (x,x') R (y,y') \implies (x',x) R (y',y) \wedge (x,x) R (y,y).
\end{align}
 
Hence the category $\Splitd[\CPM[\Rel]]$ has pairs $(X,\eq)$ as
objects, where $X$ is a set, and $\eq$ is a partial equivalence
relation on $X \times X$ satisfying 
\begin{align}\label{eq:split}
  (x,x') \sim (y,y') & \implies (x',x) \sim (y',y) \wedge (x,x) \sim (y,y);
\end{align}
morphisms $(X,\eq) \to (Y,\Eq)$ are relations $R \colon X \times X \to Y \times Y$ satisfying~\eqref{eq:cpmrel} and $R= \Eq \circ R \circ \eq$.

In this description, $F(\cat{G}) = (\Mor(\cat{G}),\eq)$, where $(a,b)
\sim (c,d)$ if and only if $a^{-1} b = c^{-1} d$ (and both
compositions are well-defined). 

When speaking about a partial equivalence relation $\eq$ on $X \times
X$, we will abbreviate $[(x,x')]_{\eq}$ to $[x,x']_{\eq}$.

\begin{lemma}\label{splitisos}
  Objects $(X,\eq)$ and $(Y,\Eq)$ are isomorphic in $\Splitd[\Rel]$ if
  and only if $\Dom(\eq) / \eq$ and $\Dom(\Eq) / \Eq$ are isomorphic
  in $\Rel$ (and hence in $\cat{Set}$).   
  Furthermore, partial equivalence relations $\eq \colon A \times A
  \to A \times A$ and $\Eq \colon B \times B \to B \times B$ are
  isomorphic objects of $\Splitd[\CPM[\Rel]]$ if and only if there is a
  bijection $\alpha \colon \Dom(\eq)/\eq \to \Dom(\Eq)/\Eq$ satisfying 
  \begin{align}\label{eq:quotients}
    \alpha( [a,a']_{\eq} ) = [b,b']_{\Eq} \implies \alpha( [a',a]_{\eq} ) = [b',b]_{\eq} \wedge \alpha([a,a]_{\eq}) = [b,b]_{\Eq}.
  \end{align}
\end{lemma}
\begin{proof}
  Suppose $(X,\eq)$ and $(Y,\Eq)$ are isomorphic in $\Splitd[\Rel]$. Say relations $R \colon X \to Y$ and $S \colon Y \to X$ satisfy $S \circ R = \eq$ and $R \circ S = \Eq$. Define relations $U \colon \Dom(\eq) / \eq \to \Dom(\Eq) / \Eq$ and $V \colon \Dom(\Eq) / \Eq \to \Dom(\eq) / \eq$ by $U=\{([x]_{\eq},[y]_{\Eq}) \mid (x,y) \in R\}$ and $V=\{([y]_{\Eq},[x]_{\eq}) \mid (y,x) \in S\}$. Then
  \[
    V \circ U = \{([x]_{\eq},[x]_{\eq}) \mid (x,x') \in S \circ R\} = \{([x]_{\eq},[x]_{\eq}) \mid x \in \Dom(\eq) \} = \id[\Dom(\eq) / \eq],
  \]
  and similarly $U \circ V = \id[\Dom(\Eq) / \Eq]$. Hence $\Dom(\eq)/\eq$ and $\Dom(\Eq)/\Eq$ are isomorphic in $\Rel$.

  Conversely, assume that $\Dom(\eq)/\eq$ and $\Dom(\Eq)/\Eq$ are isomorphic in $\Rel$. Say  $U \colon \Dom(\eq) / \eq \to \Dom(\Eq) / \Eq$ and $V \colon \Dom(\Eq) / \Eq \to \Dom(\eq) / \eq$ satisfy $U \circ V = \id[\Dom(\Eq) / \Eq]$ and $V \circ U = \id[\Dom(\eq) / \eq]$. Define relations $R \colon X \to Y$ and $S \colon Y \to X$ by $R = \{ (x,y) \mid [x]_{\eq} U [y]_{\Eq} \}$ and $S = \{(y,x) \mid [y]_{\Eq} V [x]_{\eq} \}$. Then 
  \[
    \Eq \circ R \circ \eq = \{ (x,y) \mid \exists x',y'\colon x \sim x', y \approx y', [x']_{\eq} U [y']_{\approx} \} = R,
  \]
  and similarly $S$ is a well-defined morphism of $\Splitd[\Rel]$.
  Also
  \[
    S \circ R = \{(x,x') \mid \exists y\colon [x]_{\eq} U [y]_{\Eq} V [x']_{\eq} \} = \{(x,x') \mid ([x]_{\eq},[x']_{\eq}) \in \id[\Dom(\eq)/\eq] \} = \eq,
  \]
  and similarly $R \circ S = \Eq$. So $(X,\eq)$ and $(Y,\Eq)$ are isomorphic in $\Splitd[\Rel]$.

  In case $X=A \times A$ and $Y=B \times B$, notice that $R$ and $S$  satisfy~\eqref{eq:cpmrel} if and only if the bijection $\alpha=U$ and its inverse $\alpha^{-1}=V$ satisfy~\eqref{eq:quotients}. Finally, $\alpha^{-1}$ satisfies~\eqref{eq:quotients} precisely when $\alpha$ does.
\end{proof}

\begin{lemma}\label{quotientFG}
  If $\cat{G}$ is a small groupoid and
  $F(\cat{G})=(\Mor(\cat{G}),\eq)$, then $\Dom(\eq)/\eq$ is in
  bijection with $\Mor(\cat{G})$. 
  Furthermore, an object $(X,\eq)$ of $\Splitd[\CPM[\Rel]]$ is isomorphic to $F(\cat{G})$ for a small groupoid $\cat{G}$ if and only if there is a bijection $\beta \colon \Mor(\cat{G}) \to \Dom(\eq)/\eq$ satisfying
  \begin{align}\label{eq:imageF}
    \beta(g) = [x,x']_{\eq} \implies \beta(g^{-1}) = [x',x]_{\eq} \wedge \beta(\id[\dom(g)]) = [x,x]_{\eq}
  \end{align}
  for all $g \in \Mor(\cat{G})$.
\end{lemma}
\begin{proof}
  Define functions $\gamma \colon \Dom(\eq)/\eq \to \Mor(\cat{G})$ and
  $\beta \colon \Mor(\cat{G}) \to \Dom(\eq)/\eq$ by
  $\gamma([g,f]_{\eq}) = g^{-1} f$ and $\beta(h) =
  [\id[\cod(h)],h]_{\eq}$. Then  
  \[
    \beta \circ \gamma([g,f]_{\eq}) = \gamma(g^{-1}f) = [\id[\dom(g)],g^{-1}f]_{\eq} = [g,f]_{\eq}
  \]
  and $\gamma \circ \beta(h) = h$.
  The second statement now follows from Lemma~\ref{splitisos}.
\end{proof}

\begin{theorem}
  The functor $F \colon \CPs[\Rel] \to \Split{\I}[\CPM[\Rel]]$ is not an equivalence.
\end{theorem}
\begin{proof}
  In the setting of the second statement of Lemma~\ref{quotientFG},
  the identities of $\cat{G}$ must be the morphisms
  $\beta^{-1}([x,x]_{\eq})$ for $x \in X$. Therefore we may restrict to groupoids with
  $\Ob(\cat{G}) = \{ [x,x]_{\eq} \mid x \in X\}$. Furthermore, it then
  follows from~\eqref{eq:imageF} 
  that $\beta^{-1}[x,x']_{\eq}$ is a morphism $[x,x]_{\eq} \to
  [x',x']_{\eq}$.
  The same counterexample as in the proof of
  Theorem~\ref{thm:FRelnotequivalence} now shows that $F$ is not an
  equivalence. Take $X=\{0,1,2\}$, and let $\eq$ be specified by
  \begin{align*}
    & (0,0) \sim (0,0), \qquad (1,1) \sim (1,1), \qquad (2,2) \sim (2,2), \\
    & (0,1) \sim (0,1), \qquad (1,0) \sim (1,0), \\
    & (1,2) \sim (1,2), \qquad (2,1) \sim (2,1);
  \end{align*}
  no other pairs satisfy $(x,x') \sim (y,y')$. In particular, $(0,2)
  \not\sim (0,2)$. Then $\eq$ is a partial equivalence relation that
  satisfies~\eqref{eq:split}, and so $(X,\eq)$ is a well-defined object
  in $\Splitd[\CPM[\Rel]]$. Now suppose that $(X,\eq)$ is isomorphic to
  $F(\cat{G})$. As discussed above, we may assume that $\cat{G}$ has
  three objects $0,1,2$ and seven morphisms, with types as follows.
  \[\xymatrix@C+10ex{
    [0,0]_{\eq} \ar@(ul,ur)^-{\beta^{-1}[0,0]_{\eq}} 
    \ar@<.5ex>^-{\beta^{-1}[0,1]_{\eq}}[r]
    & [1,1]_{\eq} \ar@(ul,ur)^-{\beta^{-1}[1,1]_{\eq}}
    \ar@<.5ex>^-{\beta^{-1}[1,0]_{\eq}}[l]
    \ar@<.5ex>^-{\beta^{-1}[1,2]_{\eq}}[r]
    & [2,2]_{\eq} \ar@(ul,ur)^-{\beta^{-1}[2,2]_{\eq}} 
    \ar@<.5ex>^-{\beta^{-1}[2,1]_{\eq}}[l]
  }\]
  But this can never be made into a groupoid: there are
  arrows $[0,0]_{\eq} \to [1,1]_{\eq}$ and $[1,1]_{\eq} \to
  [2,2]_{\eq}$, but no morphisms $[0,0]_{\eq} \to [2,2]_{\eq}$, so no 
  composition can be defined. We conclude that the essential image of
  $F$ is not all of $\Splitd[\CPM[\Rel]]$.

  In fact, $(X,\eq)$ is an object of $\Split{\I}[\CPM[\Rel]]$, i.e.\ it
  is unital (and therefore trace-preserving) precisely when $(x,x) \in D(\eq)$ for
  all $x \in X$. Since the above counterexample satisfies this, the
  restriction $F \colon \CPs[\Rel] \to \Split{\I}[\CPM[\Rel]]$ is not
  an equivalence.
\end{proof}

\section{Biproducts}

This section shows that if $\cat{V}$ has biproducts, then so does $\CPs[\V]$,
and there is a full and faithful functor $\CPM[\V]^\oplus \to \CPs[\V]$. Furthermore, this functor is an equivalence for $\V=\Hilb$, but not for $\V=\Rel$. 

Early in the development of categorical quantum mechanics, classical
information was modelled by biproducts. Since categories of completely
positive maps need not inherit biproducts from their base category,
biproducts had to be explicitly added to $\CPM[V]$. Later on,
Frobenius algebras were proposed as an alternative to biproducts. We now come
full circle by proving a satisfying relationship between Frobenius
algebras, completely positive maps, and biproducts.
This requires quite some detailed (matrix) calculations. We
first summarise the basic interaction of biproducts and dual
objects.

Recall that a \emph{zero object} is a terminal initial object. A zero
object induces unique \emph{zero maps} from any object to any other
object that factor through the zero object. A \emph{biproduct} of
objects $A$ and $B$ consists of an object 
$A \oplus B$ together with morphisms $\xymatrix@1{A
\ar@<.5ex>^-{i_A}[r] & A \oplus B \ar@<-.5ex>_-{p_B}[r]
\ar@<.5ex>^-{p_A}[l] & B \ar@<-.5ex>_-{i_B}[l]}$, such that $A \oplus
B$ is simultaneously a product of $A$ and $B$ with projections $p_A$
and $p_B$ and a coproduct of $A$ and $B$ with injections $i_A$ and
$i_B$, satisfying $p_A \circ i_A = \id[A]$, $p_B \circ i_B = \id[B]$,
$p_A \circ i_B = 0$, and $p_B \circ i_A = 0$. A category has
\emph{dagger biproducts} when it has a zero object and biproducts of
any pair of objects such that $p_A=i_A^\dag$ and $p_B=i_B^\dag$.

Categories with biproducts are automatically enriched over commutative
monoids: $f+g = [\id,\id] \circ (f \oplus g) \circ \langle \id, \id
\rangle$. This means that morphisms between biproducts of objects can
be handled using a matrix calculus.
We will also write $\Delta_A$ for the diagonal tuple $\langle
\id, \id \rangle = \left(\begin{smallmatrix} \id \\
    \id \end{smallmatrix}\right) \colon A \to A \oplus A$. 

In a compact category $\cat{C}$, the functor $- \otimes A \colon
\cat{C} \to \cat{C}$ is both left and right adjoint to the functor $-
\otimes A^*$. If $\cat{C}$ has a zero object, it follows directly that
$A \otimes 0 \cong 0$ for any object $A$. Consequently, if $f$ is any
morphism, then $f \otimes 0$ factors through $\dom(f) \otimes 0$ and
must therefore equal the zero morphism.

The adjunctions also imply that $- \otimes A$ preserves both limits
and colimits. So if $\cat{C}$ has biproducts, then $\otimes$
distributes over $\oplus$. Consequently, the following morphisms are
each other's inverse. 
  \[\xymatrix@R+7ex{
    (A \oplus B) \otimes (C \oplus D)
    \ar@<1ex>[d]^-{\begin{pmatrix} p_A \otimes p_C \\ p_A \otimes p_D
        \\ p_B \otimes p_C \\ p_B \otimes p_D \end{pmatrix}}
    \\
    (A \otimes C) \oplus (A \otimes D) \oplus (B \otimes C) \oplus (B \otimes D)
    \ar@<1ex>[u]^-{\begin{pmatrix} i_A \otimes i_C & i_A \otimes i_D
        & i_B \otimes i_C & i_B \otimes i_D \end{pmatrix}}
  }\]
It follows that $f \otimes (g+h) = (f \otimes g) + (f \otimes h)$ and
$(f+g) \otimes h = (f \otimes h) + (g \otimes h)$.
 Also
\begin{align*}
  \cat{C}((A \oplus B)^*, C)
  & \cong \cat{C}(I, (A \oplus B) \otimes C) \\
  & \cong \cat{C}(I, (A \otimes C) \oplus (B \otimes C)) \\
  & \cong \cat{C}(I, A \otimes C) \times \cat{C}(I, B \otimes C) \\
  & \cong \cat{C}(A^*, C) \times \cat{C}(B^*, C) \\
  & \cong \cat{C}(A^* \oplus B^*, C),
\end{align*}
so by the Yoneda lemma $(A \oplus B)^* \cong A^* \oplus B^*$.
Tracing through the steps carefully, we may in fact choose the
following unit and counit for compactness:
\begin{align*}
  \varepsilon_{A \oplus B} & = (\varepsilon_A \circ (p_A \otimes
  p_{A^*})) + (\varepsilon_B \circ (p_B \otimes p_{B^*})) \colon
  (A\oplus B) \otimes (A^* \oplus B^*) \to I, \\
  \eta_{A \oplus B} & = ((i_{A^*} \otimes i_A) \circ \eta_A) +
  ((i_{B^*} \otimes i_B) \circ \eta_B) \colon I \to (A^*\oplus B^*)
  \otimes (A \oplus B).
\end{align*}

\begin{lemma}\label{lem:nfabiprod}
  If $(A,m_A,u_A)$ and $(B,m_B,u_B)$ are normal dagger Frobenius
  algebras in a dagger compact category with dagger biproducts, then
  \begin{align*}
    m_{A \oplus B} &= \begin{pmatrix} m_A \circ (p_A \otimes p_A) \\
      m_B \circ (p_B \otimes p_B) \end{pmatrix} \colon (A \oplus B)
    \otimes (A \oplus B) \to (A \oplus B) \\
    u_{A \oplus B} &= \begin{pmatrix} u_A \\ u_B \end{pmatrix} \colon
    I \to A \oplus B
  \end{align*}
  make $A\oplus B$ into a normal dagger Frobenius
  algebra. Furthermore, $0$ is uniquely made into a normal
  dagger Frobenius algebra by
  \[
    m_0 = 0 \colon 0 \otimes 0 \to 0, 
    \qquad
    u_0 = 0 \colon I \to 0.
  \]
\end{lemma}
\begin{proof}
  Verifying the required properties is a matter of equational
  rewriting of matrices. For example, to show that $(A \oplus B, m_{A
    \oplus B}, u_{A \oplus B})$ is normal:
  \begin{align*}
    & \varepsilon_{A \oplus B} \circ (\id[A^* \oplus B^*] \otimes m_{A
      \oplus B}) \circ (\eta_{A \oplus B} \otimes \id[A \oplus B])
    \\
    & = [ \varepsilon_A \circ (\id[A^*] \otimes m_A)
      \circ (\eta_A \otimes \id[A]) , \varepsilon_B \circ
      (\id[B^*] \otimes m_B) \circ (\eta_B \otimes \id[B]) ] \\
    & = [ u_A^\dag , u_B^\dag ]
    = u_{A \oplus B}^\dag.
  \end{align*}
  One similarly verifies associativity and the Frobenius law. Because
  $\V$ is compact, unitality then follows automatically~\cite[Proposition~7]{abramskyheunen:hstar}.
  As for $(0,m_0,u_0)$: 
  all required diagrams commute because they are in fact equal to the
  zero morphism, and hence the multiplication $m_0$ is unique.
\end{proof}

\begin{theorem}\label{thm:biproducts}
  If $\V$ is dagger compact with dagger biproducts, so is $\CPs[\V]$.
\end{theorem}
\begin{proof}
  Because dagger biproducts are preserved under dagger equivalences,
  it suffices to prove that $\CPsn[\V]$ has dagger biproducts by Remark~\ref{rem:CPsn}.
  We claim that the objects defined in Lemma~\ref{lem:nfabiprod} in
  fact form dagger biproducts in $\CPsn[\V]$. We prove this according
  to the following strategy. Any object $A$ of $\CPsn[\V]$ gives
  morphisms $0 \colon A \to 0$, $i_A \colon A \to A \oplus 0$, $p_A
  \colon A \oplus 0 \to A$, and $\Delta_A \colon A \to A \oplus A$ in
  $\V$. We will show that these are all  
  *-ho\-mo\-mor\-phisms (see~\cite[Definition~3.6]{coeckeheunenkissinger:cpstar}), and
  hence well-defined morphisms in $\CPsn[\V]$ 
  by~\cite[Lemma~3.7]{coeckeheunenkissinger:cpstar}. 
  Furthermore, it is easy to see that if $f$
  and $g$ are morphisms in $\CPsn[\V]$, then so is $f \oplus g$.
  Observing that all coherence isomorphisms for $(\V,\oplus,0)$ and
  their inverses are built by composition from the above maps and
  their daggers, these are also well-defined morphisms in $\CPsn[\V]$.
  Thus we may conclude that $\CPsn[\V]$ has a symmetric monoidal
  structure $(\oplus,0)$, under which every object has a unique comonoid
  structure. Therefore, the monoidal product is in fact a
  product~\cite[Theorem~2.1]{heunen:semimoduleenrichment}. Moreover,
  because $\CPsn[\V]$ is compact by
  \cite[Theorem~3.4]{coeckeheunenkissinger:cpstar}, products are
  biproducts~\cite{houston:biproducts}. Finally, these 
  biproducts are dagger biproducts because they are so in $\V$.

  First consider $0 \colon A \to 0$. Regardless of the multiplication
  $m_A$, the morphism $0 \colon (A,m_A) \to (0,m_0)$ is trivially a
  *-homomorphism. 

  Next, consider $i_A \colon A \to A \oplus
  0$. Lemma~\ref{lem:nfabiprod} shows 
  $
     m_{A \oplus 0} 
     = i_A \circ m_A \circ (p_A \otimes p_A)$
  and $
     u_{A \oplus 0}
     = i_A \circ u_A$.
  Therefore $m_{A \oplus
    0} \circ (i_A \otimes i_A) = i_A \circ m_A \colon A \otimes A \to
  A \oplus 0$.
  Writing 
  $
  s_A = \lambda_A
  \circ (\varepsilon_A \otimes \id[A])
  \circ (\id[A^*] \otimes m_A^\dag)
  \circ (\id[A^*] \otimes u_A)
  \circ \rho_{A^*}^{-1} \colon A^* \to A
  $
  for the involution and $\lambda_A \colon I \otimes A \to A$ and
  $\rho_A \colon A \otimes I \to A$ for the coherence isomorphisms, 
  one can verify that
  $s_{A \oplus 0} = s_A \oplus 0 \colon A^* \oplus 0 \to A \oplus 0$.
  Hence $s_{A \oplus 0} \circ (i_A)_* = s_{A \oplus 0} \circ i_{A^*} =
  i_A \circ s_A$, making $i_A$ into a *-homomorphism. 

  As to $p_A \colon A \oplus 0 \to A$, the above gives $p_A \circ s_{A \oplus
    0} = s_A \circ p_{A^*} = s_A \circ (p_A)_*$ and $p_A \circ m_{A
    \oplus 0} = m_A \circ (p_A\otimes p_A)$. Hence also $p_A$ is a *-homomorphism.

  Finally, we turn to $\Delta_A \colon A \to A \oplus A$. 
  It follows from Lemma~\ref{lem:nfabiprod} that
  $m_{A \oplus A} \circ (\Delta_A \otimes \Delta_A)
  = \Delta_A \circ m_A$. 
  Furthermore, one verifies $s_{A \oplus A} \circ (\Delta_A)_*
  = \Delta_A \circ s_A$.
  So $\Delta$ is a *-homomorphism, completing the proof.
\end{proof}

Write $\cat{C}^\oplus$ for the biproduct completion of a category
$\cat{C}$.

\begin{corollary}
  If $\V$ is a dagger compact category with dagger biproducts, there
  is a full and faithful functor $\CPM[\V]^\oplus \to \CPs[\V]$.
\end{corollary}
\begin{proof}
  \cite[Theorem~4.3]{coeckeheunenkissinger:cpstar} gives a full and
  faithful functor $\B \colon \CPM[\V] 
  \to \CPs[\V]$. Theorem~\ref{thm:biproducts} shows that $\CPs[\V]$
  has biproducts. Thus the universal property of $\CPM[\V]^\oplus$
  guarantees that $L$ lifts to a functor $\CPM[\V]^\oplus \to
  \CPs[\V]$ that is full and faithful.
\end{proof}

\begin{example}
  By~\cite[Proposition~3.5]{coeckeheunenkissinger:cpstar},
  $\CPs[\FHilb]$ is the 
  category of finite-dimensional C*-algebras and completely positive
  maps. Similarly, $\CPM[\FHilb]$ can be identified with the full
  subcategory of finite-dimensional C*-\emph{factors}, i.e. matrix
  algebras $\M_n$. Because any finite-dimensional C*-algebra is a
  direct sum of matrix
  algebras~\cite[Theorem~III.1.1]{davidson:cstar}, the functor of the
  previous corollary is an equivalence between the categories 
  $\CPs[\FHilb]$ and $\CPM[\FHilb]^\oplus$.
\end{example}

\begin{example}
  By~\cite[Proposition~5.3]{coeckeheunenkissinger:cpstar}, $\CPs[\Rel]$ is the
  category of (small) groupoids and relations respecting
  inverses. Similarly,
  by~\cite[Proposition~5.4]{coeckeheunenkissinger:cpstar},
  $\CPM[\Rel]$ can be identified with the full subcategory of
  indiscrete (small) groupoids. But there exist groupoids that are not
  isomorphic to a disjoint union of indiscrete ones in $\CPs[\Rel]$. 
  For example, groupoids isomorphic to $\mathbb{Z}_2$ in
  $\CPs[\Rel]$ must have a single object and two
  morphisms, and therefore cannot be a disjoint union of indiscrete groupoids.
  Hence the functor $\CPM[\Rel]^\oplus \to \CPs[\Rel]$ of the previous
  corollary is not an equivalence, and in fact there cannot be an
  equivalence between $\CPM[\Rel]^\oplus$ and $\CPs[\Rel]$.
\end{example}

\bibliographystyle{eptcs}
\bibliography{cpproj}

\end{document}